\RequirePackage{fix-cm}
\documentclass[smallextended, numbook]{svjour3}       
\smartqed  
\usepackage{graphicx}
\usepackage{latexsym}
\usepackage{amssymb}
\usepackage{amsmath}
\usepackage{relsize}
\usepackage{url}
\begin{document}
\newcommand{\nc}{\newcommand}
\nc{\nt}{\newtheorem}
\nt{thm}{Theorem}[section]
\nt{cor}[thm]{Corollary}
\nt{prop}[thm]{Proposition}
\nt{obs}[thm]{Observation}
\nt{lem}[thm]{Lemma}
\nt{defn}[thm]{Definition}
\nt{exa}[thm]{Example}
\nt{rem}[thm]{Remark}
\nt{ass}[thm]{Assumption}
\nt{alg}[thm]{Algorithm}
\nt{con}[thm]{Conjecture}
\nc{\ip}[2]{\mbox{$\langle #1,#2 \rangle$}}
\nc{\linespace}{\vspace{\baselineskip} \noindent}
\nc{\R}{{\bf R}}
\nc{\cl}{\mbox{\rm cl}\,}
\nc{\cls}{ \mbox{{\scriptsize {\rm cl}}}\,}
\nc{\conv}{\mbox{\rm conv}}
\nc{\rb}{\mbox{\rm rb}\,}
\nc{\ri}{\mbox{\rm ri}\,}
\nc{\inter}{\mbox{\rm int}\,}
\nc{\kernel}{\mbox{\rm ker}\,}
\nc{\bd}{\mbox{\rm bd}\,}
\nc{\spann}{\mbox{\rm span}\,}
\nc{\rint}{\mbox{\rm rint}\,}
\nc{\epi}{\mbox{\rm epi}\,}
\nc{\gph}{\mbox{\rm gph}\,}
\nc{\rge}{\mbox{\rm rge}\,}
\nc{\rgel}{\mbox{\rm {\scriptsize rge}}\,}
\nc{\sepi}{\mbox{\rm {\scriptsize epi}}\,}
\nc{\sbd}{\mbox{\rm {\scriptsize bd}}\,}
\nc{\dom}{\mbox{\rm dom}\,}
\nc{\sreg}{\mbox{\rm sreg}\,}
\nc{\lin}{\mbox{\rm lin}\,}
\nc{\detr}{\mbox{\rm det}\,}
\nc{\para}{\mbox{\rm par}\,}
\nc{\crit}{\mbox{\rm crit}\,}
\nc{\cone}{\mbox{\rm cone}\,}
\nc{\diag}{\mbox{\rm Diag}\,}
\nc{\fix}{\mbox{\rm Fix}}
\nc{\rank}{\mbox{\rm rank}\,}
\nc{\co}{\mbox{\rm co}\,}
\nc{\cco}{\overline{\mbox{\rm co}}\,}

\newcommand{\argmin}{\operatornamewithlimits{argmin}}
\newcommand{\argmax}{\operatornamewithlimits{argmax}}
\newcommand{\lf}{\operatornamewithlimits{liminf}}

\title{Tilt stability, uniform quadratic growth, and strong metric regularity of the subdifferential.\thanks{Work of Dmitriy Drusvyatskiy on this paper has been partially supported by the NDSEG grant from the Department of Defence. Work of A. S. Lewis has been supported in part by National Science Foundation Grant DMS-0806057 and by the US-Israel Binational Scientific Foundation Grant 2008261.}
}

\titlerunning{Tilt stability, uniform quadratic growth, and strong metric regularity}        

\author{D. Drusvyatskiy         \and
        A. S. Lewis 
}


\institute{D. Drusvyatskiy \at
              School of Operations Research and Information Engineering,
    Cornell University,
    Ithaca, New York, USA; \\
              Tel.: (607) 255-4856\\
              Fax: (607) 255-9129\\
              \email{dd379@cornell.edu}\\
              {\tt http://people.orie.cornell.edu/{\raise.17ex\hbox{$\scriptstyle\sim$}}dd379/}.           
           \and
           A. S. Lewis \at
              School of Operations Research and Information Engineering,
    Cornell University,
    Ithaca, New York, USA;\\
    Tel.: (607) 255-9147 \\
    Fax:  (607) 255-9129\\
    \email{aslewis@orie.cornell.edu}\\
    {\tt http://people.orie.cornell.edu/{\raise.17ex\hbox{$\scriptstyle\sim$}}aslewis/}.
}

\date{Received: date / Accepted: date}

\maketitle

\begin{abstract}
We prove that uniform second order growth, tilt stability, and strong metric regularity of the limiting subdifferential --- three notions that have appeared in entirely different settings --- are all essentially equivalent for any lower-semicontinuous, extended-real-valued function. 

\end{abstract}

\section{Introduction}
\label{intro}
Second-order conditions are ubiquitous in non-linear optimization, in particular playing a central role in perturbation theory and in the analysis of algorithms. See for example \cite{Bon_Shap,non_opt}. Classically, a point $\bar{x}$ is called a {\em strong local minimizer} of a function $f$ on $\R^n$ if there exist $\kappa > 0$ and a neighbourhood $U$ of $\bar{x}$ such that the inequality 
$$f(x)\geq f(\bar{x})+\kappa |x-\bar{x}|^2 ~~\textrm{ holds for all } 
x \in U.$$
Here $|\cdot |$ denotes the standard euclidean norm on $\R^n$.
For smooth $f$, this condition simply amounts to positive definiteness of the Hessian $\nabla^2 f(\bar{x})$. 

Existence of a strong local minimizer is a sufficient condition for a number of desirable properties: even in classical nonlinear programming, it typically drives local convergence analysis for algorithms. However, this notion has an important drawback, namely that strong local minimizers are sensitive to small perturbations to the function. To illustrate, the origin in $\R^2$ is a strong (global) minimizer of the convex function $f(x,y)=(|x|+|y|)^2$, whereas strong minimizers cease to exist for the slightly perturbed functions $f_t(x,y)=(|x|+|y|)^2 +t(x+y)$ for any $t\neq 0$. 

In light of this instability, it is natural to look for a more robust quadratic growth condition, namely we would like the constant $\kappa$ and the neighbourhood $U$, appearing in the definition of strong local minimizers, to be uniform relative to linear perturbations of the function.  
\begin{defn}[Stable strong local minimizers]\label{def:stab}
{\rm We will say that $\bar{x}$ is a {\em stable strong local minimizer} of a function $f\colon\R^n\to\R\cup \{-\infty,+\infty\}$ if there is a constant $\kappa> 0$ and a neighbourhood $U$ of $\bar{x}$ so that for each vector $v$ near the origin, there is a a point $x_v$ (necessarily unique) in $U$, with $x_{0}=\bar{x}$, so that in terms of the perturbed functions $f_v:=f(\cdot)-\langle v,x\rangle$, the inequality 
$$f_v(x)\geq f_v(x_v)+\kappa |x-x_v|^2 \textrm{ holds } \textrm{for each } x \textrm{ in } U.$$}
\end{defn}
This condition appears under the name of uniform quadratic growth for tilt perturbations in \cite{non_opt}, where it is considered in the context of optimization problems having a particular presentation. One could go further and require the dependence $v\mapsto x_v$ to be Lipschitz continuous, though it is easy to see that this requirement is automatically satisfied whenever $\bar{x}$ is  a stable strong local minimizer (see Proposition~\ref{lip}). 

In the variational-analytic literature, conditioning and sensitivity of optimization problems is deeply tied to the notion of {\em metric regularity} \cite{imp,met_sub,VA,Mord_1}.  For us, the work of Artacho-Geoffroy \cite{artacho} in this area will be particularly important. There the authors considered regularity properties of the workhorse of convex analysis, the convex subdifferential mapping $x\mapsto \partial f(x)$, 
and fully characterized such properties in terms of a variety of quadratic growth conditions. Results of the same flavour also appear in \cite{Bon_Shap}. In this short note, we will generalize the equivalence \cite[Corollary 3.9]{artacho} to {\em all} lower-semicontinuous functions possessing a natural continuity property. Consequently, we will show that for such a function $f$ on $\R^n$ and a local minimizer $\bar{x}$ of $f$, the {\em limiting subdifferential} $\partial f$ is {\em strongly metrically regular} at $(\bar{x},0)$ if and only if $f$ is prox-regular at $\bar{x}$ for $0$ (see Definition~\ref{defn:prox_glob}) and $\bar{x}$ is a stable strong local minimizer of $f$. 

Our proof strategy is straightforward. We will try to reduce the general situation to the convex case, thereby allowing us to apply \cite[Corollary 3.9]{artacho}. The key step in this direction is to observe that stable strong local minimizers are {\em tilt-stable}, in the sense of \cite{PR}. In fact, employing a reduction to the convex case, we will establish the surprising equivalence: stable strong local minimizers are one and the same as tilt-stable local minimizers (Corollary~\ref{cor:eqv}). We should also note that our results generalize \cite[Theorem 6.3]{shan}, which is only applicable to {\em ${\bf C}^2$-partly smooth} functions possessing  a certain nondegeneracy condition. 

As a by-product of our work, we will deduce that there is a complete characterization of stable strong local minimizers using positive definiteness of Mordukhovich's generalized Hessian $\partial^2 f(\bar{x}|0)$. For more details on the theory of generalized Hessians see \cite{Mord_1,Mord_2}. This is significant since there is now a fairly effective calculus of this second-order nonsmooth object \cite{sorder}, thereby providing the means of identifying stable strong local minimizers in many instances of practical importance.

\section{Preliminaries}
In this section, we summarize some of the fundamental tools used in variational analysis and nonsmooth optimization.
We refer the reader to the monographs Borwein-Zhu \cite{Borwein-Zhu}, Clarke-Ledyaev-Stern-Wolenski \cite{CLSW}, Mordukhovich \cite{Mord_1,Mord_2}, and Rockafellar-Wets \cite{VA}, for more details.  Unless otherwise stated, we follow the terminology and notation of \cite{VA}.

The functions that we will be considering will take their values in the extended real line $\overline{\R}:=\R\cup\{-\infty,\infty\}$. For a function $f\colon\R^n\rightarrow\overline{\R}$, the {\em domain} of $f$ is $$\mbox{\rm dom}\, f:=\{x\in\R^n: f(x)<+\infty\},$$ and the {\em epigraph} of $f$ is $$\mbox{\rm epi}\, f:= \{(x,r)\in\R^n\times\R: r\geq f(x)\}.$$
A function $f$ is {\em lower-semicontinuous} (or {\em lsc} for short) at $\bar{x}$ if the inequality $\lf_{x\to\bar{x}} f(x)\geq f(\bar{x})$ holds. 

Throughout this work, we will only use Euclidean norms. Hence for a point $x\in\R^n$, the symbol $|x|$ will denote the standard Euclidean norm of $x$. We let ${\bf B}_{\epsilon}(\bar{x})$ be an open ball around $\bar{x}$ of radius $\epsilon$, and we let $\overline{{\bf B}}_{\epsilon}(\bar{x})$ denote its closure.

\subsection{Tilt stability}
In establishing our main result, it will be crucial to relate the notion of stable strong strong local minimizers to the theory of tilt stability, introduced in \cite{PR}.
We begin with a definition \cite[Definition 1.1]{PR}. 
\begin{defn}[Tilt stability]\label{defn:tilt}
{\rm
A point $\bar{x}$ gives a {\em tilt-stable local minimum} of the function $f\colon\R^n\to\overline{\R}$ if $f(\bar{x})$ is finite and there exists an $\epsilon>0$ such that the mapping 
$$M:v\mapsto \argmin_{|x-\bar{x}|\leq \epsilon}\{f(x)-\langle v,x \rangle\},$$
is single-valued and Lipschitz on some neighbourhood of $0$ with $M(0)=\bar{x}$.
}
\end{defn}

For ${\bf C}^2$ smooth functions, tilt stability reduces to positive-definiteness of the Hessian $\nabla^{2}f(\bar{x})$ \cite[Proposition 1.2]{PR}. We will see (Corollary~\ref{cor:eqv}) that the notions of tilt stability and stable strong local minimality are the same for {\em all} lsc functions --- a rather surprising result. As a first step in establishing this equivalence, we now show that stable strong local minimizers depend in a Lipschitz way on the perturbation parameters.
\begin{prop}[Lipschitzness of stable strong local minimizers]\label{lip}{\ \\}
Consider a lsc function $f\colon\R^n\to\overline{\R}$ and suppose that $\bar{x}$ is a stable strong local minimizer of $f$. Then the correspondence $v\mapsto x_v$ of  Definition~\ref{def:stab} is locally Lipschitz around $0$. 
\end{prop}
\begin{proof}
There is a constant $\kappa$ and a neighbourhood $U$ of $\bar{x}$ so that for any vectors $v,w$ near the origin, we have
\begin{align*}
f(x_w)&\geq f(x_v)+\langle v,x_w- x_v\rangle +\kappa |x_v- x_w|^2,\\ 
f(x_v)&\geq f(x_w)+\langle w,x_v- x_w\rangle +\kappa |x_v- x_w|^2.
\end{align*}
Adding the two inequalities and dividing by $|x_v- x_w|^2$, we obtain
$$\Big\langle \frac{v-w}{|x_v-x_w|},\frac{x_v-x_w}{|x_v-x_w|} \Big\rangle\geq 2\kappa.$$ 
We deduce $|x_v-x_w|\leq \frac{1}{2\kappa}|v-w|$, thereby establishing the result.
\qed
\end{proof}

The following is now immediate.
\begin{prop}[Stable strong local minimizers are tilt-stable]\label{prop:imp}{\ \\}
Consider a lsc function $f\colon\R^n\to\overline{\R}$ and a point $\bar{x}\in\R^n$. If $\bar{x}$ is a stable strong local minimizer of $f$, then $\bar{x}$ gives a tilt-stable local minimum of $f$.
\end{prop}
\begin{proof}
This readily follows from definition of tilt stability and Proposition~\ref{lip}.
\qed
\end{proof}
The converse of the proposition above will take some more effort to prove. We will take this up in Section~\ref{sec:main}.

\subsection{Some convex analysis}
For any set $Q\subset\R^n$, the symbol $\co Q$ will denote the convex hull of $Q$, while $\cco Q$ will denote the closed convex hull of $Q$. 
Consider any function $f\colon\R^n\to\overline{\R}$ that is minorized by some affine function on $\R^n$.
Then the set 
$\cco (\epi f)$ is an epigraph of a lsc, convex function, which we denote by $\cco f$. In some cases $\co (\epi f)$ is itself a closed set, and in such an instance, we refer to $\cco f$ simply as $\co f$.

For any (not necessarily convex) function $f\colon\R^n\to\overline{\R}$, the {\em convex subdifferential} of $f$ at $\bar{x}$, denoted by $\partial_{co} f(\bar{x})$, consists of all vectors $v$ satisfying $$f(x)\geq f(\bar{x})+\langle v,x-\bar{x}\rangle \textrm{ for all } x\in\R^n.$$ Equivalently, a vector $v$ lies in $\partial_{co} f(\bar{x})$ if and only if $\bar{x}$ is a global minimizer of the tilted function $x\mapsto f(x)-\langle v,x\rangle$.
It will be important for us to understand the relationship between the convex subdifferential of a function $f$ and the convex subdifferential of its convexification $\cco f$.
The following result will be especially important.
\begin{lem}\cite[Proposition 1.4.3]{HU}\label{lem:conv}
Consider a lsc function $f\colon\R^n\to\overline{\R}$. Suppose that $f$ is minorized by some affine function on $\R^n$. Then we have  
$$(\cco f)(x)=f(x) \Longrightarrow \partial_{co}(\cco f)(x)=\partial_{co}f(x),$$
$$\partial_{co }f(x)\neq \emptyset \Longrightarrow (\cco f)(x)=f(x),$$
$$\cco f\leq f.$$
\end{lem}


The following lemma shows that under reasonable conditions, the set of minimizers of the convexified function $\cco f$ coincides with the closed convex hull of minimizers of $f$. See \cite[Remark 1.5.7]{HU} for more details.
\begin{lem}[Minimizers of a convexified function]\label{lem:argmin}{\ \\}
Consider a lsc function $f\colon\R^n\to\overline{\R}$ with bounded domain, and suppose furthermore that $f$ is minorized by some affine function on $\R^n$.  Then $\co (\epi f)$ is a closed set, and we have
\begin{equation}\label{eqn:conv}
\argmin_{x\in\R^n}\, (\co f)(x)=\co \Big(\argmin_{x\in\R^n}\, f(x)\Big).
\end{equation} 
\end{lem}


As a direct consequence, we obtain the important observation that tilt-stable minimizers are preserved under ``local'' convexification.
\begin{prop}[Tilt-stable minimizers under convexification]\label{prop:tilt_conv}{\ \\}
Consider a lsc function $f\colon\R^n\to\overline{\R}$ and suppose that a point $\bar{x}\in\R^n$ gives a tilt-stable local minimum $f$. 
Then for all sufficiently small $\epsilon >0$, in terms of the function  $g:=f+\delta_{{\bf \overline{B}}_{\epsilon}(\bar{x})}$, we have
$$\argmin_{|x-\bar{x}|\leq\epsilon}\{f(x)-\langle v,x\rangle\}=\argmin_{x\in\R^n}\{(\co g)(x)-\langle v,x\rangle\},$$
for all $v$ sufficiently close to $0$.
Consequently $\bar{x}$ gives a tilt-stable local minimum of the convexified function $\co (f+\delta_{{\bf \overline{B}}_{\epsilon}(\bar{x})})$.
\end{prop}
\begin{proof}
By definition of tilt stability, we have that $f(\bar{x})$ is finite, and for all sufficiently small $\epsilon>0$, the mapping 
$$M:v\mapsto \argmin_{|x-\bar{x}|\leq \epsilon}\{f(x)-\langle v,x \rangle\},$$
is single-valued and Lipschitz on some neighbourhood of $0$ with $M(0)=\bar{x}$.
Letting $g:=f+\delta_{{\bf \overline{B}}_{\epsilon}(\bar{x})}$ and applying Lemma~\ref{lem:argmin} to the function $x\mapsto g(x)-\langle v,x\rangle$, we deduce
\begin{align*}
M(v)&=\argmin_{x\in\R^n}\{g(x)-\langle v,x\rangle\}=\co\big(\argmin_{x\in\R^n}\{g(x)-\langle v,x\rangle\}\big)= \\
&=\argmin_{x\in\R^n}\{\co (g(\cdot)-\langle v,\cdot\rangle)(x)\}=\argmin_{x\in\R^n}\{(\co g)(x)-\langle v,x\rangle\},
\end{align*}
for all $v$ sufficiently close to $0$. The result follows. \qed
\end{proof}


\subsection{Variational analysis preliminaries}
A {\em set-valued mapping} $G$ from $\R^n$ to $\R^m$, denoted by $G\colon\R^n\rightrightarrows\R^m$, is a mapping from $\R^n$ to the power set of $\R^m$. Thus for each  point $x\in\R^n$, $G(x)$ is a subset of $\R^m$. The {\em graph} of $G$ is defined to be 
$$\mbox{\rm gph}\, G:=\{(x,y)\in\R^n\times\R^m:y\in G(x)\}.$$

Often we will be interested in restricting both the domain and the range of a set-valued mapping. Hence for a set-valued mapping $F\colon\R^n\rightrightarrows\R^m$, and neighbourhoods $U\subset\R^n$ and $V\subset\R^m$, we define the {\em localization of} $F$ {\em relative to} $U$ {\em and} $V$ to simply be the set-valued mapping $\widehat{F}\colon\R^n\rightrightarrows\R^m$ whose graph is $(U\times V)\cap\gph F$.

For a set $S\subset \R^n$, the distance of a point $x$ to $S$ is 
$$d(x,S)=\inf\{|x-s|:s\in S\}.$$ 
We define the indicator function of $S$, denoted by $\delta_S$, to be identically zero on $S$ and $+\infty$ elsewhere.
A central notion in set-valued and variational analysis that we explore in this work is strong metric regularity.

\begin{defn}[Metric regularity]
{\rm A mapping $F\colon\R^n\rightrightarrows\R^m$ is said to be {\em strongly metrically regular} at $\bar{x}$ for $\bar{v}$, where $\bar{v}\in F(\bar{x})$, if there exist neighbourhoods $U$ of $\bar{x}$ and $V$ of $\bar{v}$ so that the localization of $F^{-1}$ relative to $V$ and $U$ defines a (single-valued) Lipschitz continuous mapping.
}
\end{defn}
This condition plays a central role in stability theory since it guarantees that near $\bar{x}$, there is a unique solution of the inclusion 
$$y\in F(x),$$ which furthermore varies in a Lipschitz way relative to perturbations in the left-hand-side. For other notions related to metric regularity, we refer the interested reader to the recent monograph \cite{imp}.

In the current work, we will use, and subsequently generalize beyond convexity, the following result that has appeared as \cite[Theorem 3.10]{artacho}.
\begin{thm}[Strong regularity of the convex subdifferential]\label{thm:artacho}{\ \\}
Consider a lsc, convex function $f\colon\R^n\to\overline{\R}$ and a point $\bar{x}$ in $\R^n$. Then the following are equivalent 
\begin{enumerate}
\item $\partial_{co} f$ is strongly metrically regular at $(\bar{x},0)$.
\item\label{it:2} There exists $\kappa >0$ and neighbourhoods $U$ of $\bar{x}$ and $V$ of $0$ so that the localization of $(\partial_{co} f)^{-1}$ relative to $V$ and $U$ is single-valued and we have
$$f(x)\geq f(\tilde{x})+\langle \tilde{v},x-\tilde{x}\rangle+\kappa |x-\tilde{x}|^2~~ \textrm{ for all } x\in U,$$
and all $(\tilde{x},\tilde{v})\in (U\times V)\cap\gph \partial_{co} f$.
\end{enumerate}
\end{thm} 
Clearly, for convex functions, property~\ref{it:2} in the theorem above is equivalent to $\bar{x}$ being a stable strong local minimizer of $f$.

We now consider subdifferentials, which are the fundamental tools in the study of general nonsmooth functions.
\begin{defn}[Subdifferentials]
{\rm
Consider a function $f\colon\R^n\to\overline{\R}$ and a point $\bar{x}$ with $f(\bar{x})$ finite. The {\em proximal subdifferential} of $f$ at $\bar{x}$, denoted by $\partial_P f(\bar{x})$, consists of all vectors $v\in\R^n$ for which there exists $r>0$ satisfying 
$$f(x)\geq f(\bar{x}) +\langle v,x-\bar{x} \rangle -r|x-\bar{x}|^2 \textrm{ for all } x \textrm{ near } \bar{x}.$$ On the other hand, the {\em limiting subdifferential} of $f$ at $\bar{x}$, denoted by $\partial f(\bar{x})$, consists of all vectors $v$ for which there exists a sequence $(x_i,f(x_i),v_i)\to (\bar{x},f(\bar{x}),\bar{v})$, with $v_i\in\partial_P f(x_i)$ for each index $i$.} 
\end{defn}

The need for the limiting construction $\partial f$ arises due to bad closure properties of the set-valued mapping $x\mapsto \{f(x)\}\times \partial_P f(x)$. 
For ${\bf C}^1$ smooth functions $f$ on $\R^n$, the subdifferential $\partial f(x)$ consists only of the gradient $\nabla f(x)$ for each $x\in\R^n$. For convex $f$, the proximal and the limiting subdifferentials coincide with the convex subdifferential
$\partial_{co} f(\bar{x})$. 

Seeking a kind of uniformity in parameters appearing in the definition of the proximal subdifferential, we arrive at the following \cite[Definition 1.1]{prox_reg}.
\begin{defn}[Prox-regularity]\label{defn:prox_glob}
{\rm
A function $f\colon\R^n\to\overline{\R}$ is prox-regular at $\bar{x}$ for $\bar{v}$ if $f$ is finite and locally lsc at $\bar{x}$ with $v\in\partial f(\bar{x})$, and there exist $\epsilon > 0$ and $\rho \geq 0$ such that 
$$f(x')\geq f(x)+\langle v, x'-x\rangle -\frac{\rho}{2}|x'-x|^2 \textrm{ for all } x'\in {\bf B}_{\epsilon}(\bar{x}),$$
when $v\in\partial f(x)$, $|v-\bar{v}|<\epsilon$, $|x-\bar{x}|<\epsilon, |f(x)- f(\bar{x})|<\epsilon$.
}
\end{defn}

In relating strong metric regularity of the subdifferential $\partial f$ to the functional properties of $f$, it is absolutely essential to require the function $(x,v)\mapsto f(x)$ to be continuous on $\gph \partial f$. This leads to the notion of subdifferential continuity, introduced in \cite[Definition 2.1]{prox_reg}. 
\begin{defn}[Subdifferential continuity]
{\rm We say that $f\colon\R^n\to\overline{\R}$ is {\em subdifferentially continuous at} $\bar{x}$ {\em for} $\bar{v}\in\partial f(\bar{x})$ if for any sequences $x_i\to\bar{x}$ and $v_i\to\bar{v}$, with $v_i\in\partial f(x_i)$, it must be the case that $f(x_i)\to f(\bar{x})$.}
\end{defn}

In particular, all lsc convex functions $f$, and more generally all strongly amenable functions (see \cite[Definition 10.23]{VA}), are both subdifferentially continuous and prox-regular at any point $\bar{x}\in\dom f$ for any vector $\bar{v}\in\partial f(\bar{x})$. See \cite[Proposition 13.32]{VA} for details.

Rockafellar and Poliquin characterized tilt stability in a number of meaningful ways \cite[Theorem 1.3]{PR}, with the notions of prox-regularity and subdifferential continuity playing a key role. The following is just a small excerpt from their result.
\begin{thm}[Characterization of tilt stability]\label{thm:char}
Consider a function $f\colon\R^n\to\overline{\R}$, with $0\in\partial f(\bar{x})$, and such that $f$ is both prox-regular and subdifferentially continuous at $\bar{x}$ for $\bar{v}=0$. Then the following are equivalent and imply the existence of $\epsilon >0$ such that the mapping $M$ in Definition \ref{defn:tilt} has the equivalent form $M(v)=(\partial f)^{-1}(v)\cap {\bf B}_{\epsilon}(\bar{x})$ for all $v$ sufficiently close to $0$.
\begin{enumerate}
\item The point $\bar{x}$ gives a tilt-stable local minimum of $f$.
\item There is a proper, lsc, strongly convex function $h$ on $\R^n$ along with neighbourhoods $U$ of $\bar{x}$ and $V$ of $0$ such that $h$ is finite on $U$, with $h(\bar{x})=f(\bar{x})$, and 
$$\big(U\times V\big)\cap \gph \partial f=\big(U\times V)\cap \gph \partial h.$$ 
\end{enumerate}
\end{thm}

Analysing the proof of the above theorem, much more can be said. Indeed, suppose that the set-up of the theorem holds and that $\bar{x}$ gives a tilt-stable local minimum of $f$. Thus there exists $\epsilon>0$ such that the mapping 
$$M:v\mapsto \argmin_{|x-\bar{x}|\leq \epsilon}\{f(x)-\langle v,x \rangle\},$$
is single-valued and Lipschitz continuous on some neighbourhood of $0$ with $M(0)=\bar{x}$. 
Then the convex function $h$ guaranteed to exist by Theorem~\ref{thm:char} can be chosen to simply be the convexified function 
$$h=\co(f+\delta_{\overline{{\bf B}}_{\epsilon}(\bar{x})}).$$
This observation will be important for the proof of our main result Theorem~\ref{thm:main}.

\section{Main results}\label{sec:main}
We begin this section by establishing a simple relationship between tilt stability and strong metric regularity of the subdifferential $\partial f$.
\begin{prop}[Tilt stability vs. Strong metric regularity]\label{prop:wchar}{\ \\}
Consider a lsc function $f\colon\R^n\to\overline{\R}$ and a point $\bar{x}$ that is a local minimizer of $f$. Consider the following properties.
\begin{enumerate}
\item\label{it1} The subdifferential mapping $\partial f$ is strongly metrically regular at $(\bar{x},0)$.
\item\label{it2} $f$ is prox-regular at $\bar{x}$ for $0$ and $\bar{x}$ gives a tilt-stable local minimum of $f$.
\end{enumerate}
Then the implication $\ref{it1}\Rightarrow \ref{it2}$ holds, and furthermore if $\ref{it1}$ holds, then for sufficiently small $\epsilon >0$ the mapping $M$ of Definition~\ref{defn:tilt} has the representation $$M(v)={\bf B}_{\epsilon}(\bar{x})\cap(\partial f)^{-1}(v),$$
for all $v$ sufficiently close to $0$.
The implication $\ref{it2}\Rightarrow \ref{it1}$ holds provided that $f$ is subdifferentially continuous at $\bar{x}$ for $0$.
\end{prop}
\begin{proof}
Suppose that \ref{it1} holds. Then, in particular, $\bar{x}$ is a strict local minimizer of $f$. Hence there exists $\epsilon  > 0$ satisfying
$$f(x)> f(\bar{x}) \textrm{ for all } x\in \overline{{\bf B}}_{\epsilon}(\bar{x}).$$ It is now easy to check that for all vectors $v$ sufficiently close to $0$, the sets $\argmin_{|x-\bar{x}|\leq \epsilon} \{f(x)-\langle v,x\rangle\}$ are contained in the open ball ${\bf B}_{\epsilon}(\bar{x})$. Hence by strong metric regularity we have 
$${\bf B}_{\epsilon}(\bar{x})\cap (\partial f)^{-1}(v)=\argmin_{|x-\bar{x}|\leq \epsilon}\{f(x)-\langle v,x \rangle\},$$
for all $v$ sufficiently close to $0$. It follows from the equation above and the definition of prox-regularity that $f$ is prox-regular at $\bar{x}$ for $0$. The validity of \ref{it2} is now immediate.

Suppose that $f$ is subdifferentially continuous at $\bar{x}$ for $0$ and that $\ref{it2}$ holds. Then by Theorem~\ref{thm:char}, we have $M(v)={\bf B}_{\epsilon}(\bar{x})\cap(\partial f)^{-1}(v)$, and consequently $\partial f$ is strongly metrically regular at $(\bar{x},0)$. \qed 
\end{proof}

We can now establish the converse of Proposition~\ref{prop:imp}, thereby showing that tilt-stable local minimizers and stable strong local minimizers are one and the same. 
\begin{cor}[Stable strong local minimizers and tilt stability]\label{cor:eqv}{\ \\}
For a lsc function $f\colon\R^n\to\overline{\R}$, a point $\bar{x}$ gives a tilt-stable local minimum of $f$ if and only if $\bar{x}$ is a stable strong local minimizer of $f$.
\end{cor}
\begin{proof}
The implication $\Leftarrow$ has been proven in Proposition~\ref{prop:imp}. We now argue the converse. To this end, suppose that $\bar{x}$ gives a tilt-stable local minimum of $f$. Then by Proposition~\ref{prop:tilt_conv}, for all sufficiently small $\epsilon >0$, the point $\bar{x}$ gives a tilt-stable local minimum of the convexified function $\co (f+\delta_{{\bf \overline{B}}_{\epsilon}(\bar{x})})$, and furthermore, in terms of the function  
$g:=f+\delta_{{\bf \overline{B}}_{\epsilon}(\bar{x})}$, we have
\begin{equation}\label{eq:rep}
\argmin_{x\in\R^n}\{ g(x)-\langle v,x\rangle\}=\argmin_{x\in\R^n}\{(\co g)(x)-\langle v,x\rangle\},
\end{equation}
for all $v$ sufficiently close to $0$. In light of $(\ref{eq:rep})$, we have
\begin{equation}\label{eqn:app}
v\in \partial_{co} g(x) \Leftrightarrow v\in \partial_{co}(\co g)(x),
\end{equation}
for all $x\in \R^n$ and all $v$ sufficiently close to $0$.

Observe $\co g$, being a lsc convex function, is both prox-regular and subdifferentially continuous at $\bar{x}$ for $0$. Hence applying Proposition~\ref{prop:wchar} to $\co g$, we deduce that the subdifferential $\partial_{co} (\co g)$ is strongly metrically regular at $(\bar{x},0)$. Consequently by Theorem~\ref{thm:artacho},  there exists $\kappa >0$ and neighbourhoods $U$ of $\bar{x}$ and $V$ of $0$ so that the localization of $(\partial_{co} (\co g))^{-1}$ relative to $V$ and $U$ is single-valued and we have
$$(\co g)(x)\geq (\co g)(\tilde{x})+\langle \tilde{v},x-\tilde{x}\rangle+\kappa |x-\tilde{x}|^2~~ \textrm{ for all } x\in U,$$
and all $(\tilde{x},\tilde{v})\in (U\times V)\cap\gph \partial_{co} (\co g)$. 

Shrinking $U$ and $V$, we may assume that the inclusion $U\subset {\bf B}_{\epsilon}(\bar{x})$ holds.
Combining (\ref{eqn:app}) and Lemma~\ref{lem:conv}, we deduce that for any pair $(\tilde{x},\tilde{v})\in (U\times V)\cap\gph \partial_{co} (\co g)$, we have $(\co g)(\tilde{x})=g(\tilde{x})=f(\tilde{x})$, and  that for any $x\in U$ the inequality $(\co g)(x)\leq f(x)$ holds.
The result follows. \qed
\end{proof}

With the preparation that we have done, the proof of our main result is now straightforward.
\begin{thm}[Strong metric regularity and quadratic growth]\label{thm:main}{\ \\}
Consider a lsc function $f\colon\R^n\to\overline{\R}$ that is subdifferentially continuous at $\bar{x}$ for $0$, where $\bar{x}$ is a local minimizer of $f$. Then the following are equivalent.
\begin{enumerate}
\item\label{reg} The subdifferential mapping $\partial f$ is strongly metrically regular at $(\bar{x},0)$. 
\item\label{tilt} $f$ is prox-regular at $\bar{x}$ for $0$ and $\bar{x}$ gives a tilt-stable local minimum of $f$.
\item\label{so}  There exists $\kappa >0$ and neighbourhoods $U$ of $\bar{x}$ and $V$ of $0$ so that the localization of $(\partial f)^{-1}$ relative to $V$ and $U$ is single-valued and we have
$$f(x)\geq f(\tilde{x})+\langle \tilde{v},x-\tilde{x}\rangle+\kappa |x-\tilde{x}|^2~~ \textrm{ for all } x\in U,$$
and for all $(\tilde{x},\tilde{v})\in (U\times V)\cap\gph \partial f$.
\item\label{stable} $f$ is prox-regular at $\bar{x}$ for $0$ and $\bar{x}$ is a stable strong local minimizer of $f$
\end{enumerate}
\end{thm}
\begin{proof}
The equivalence $\ref{reg}\Leftrightarrow\ref{tilt}$ was proven in Proposition~\ref{prop:wchar}. 

$\ref{tilt}\Rightarrow \ref{so}:$ Suppose $\ref{tilt}$ holds. Then by Theorem~\ref{thm:char} and the ensuing remarks, there is $\epsilon >0$ so that for the convexified function 
$h:=\co(f+\delta_{\overline{{\bf B}}_{\epsilon}(\bar{x})})$, we have 
$$\gph \partial f=\gph \partial h \textrm{ locally around } (\bar{x},0).$$ 
From the equivalence $\ref{reg}\Leftrightarrow\ref{tilt}$, we deduce that the mapping $\partial h$ is strongly metrically regular at $(\bar{x},0)$.
Applying Theorem~\ref{thm:artacho} to $h$, we deduce there exists $\kappa >0$ and neighbourhoods $U$ of $\bar{x}$ and $V$ of $0$ so that the localization of $(\partial f)^{-1}$ relative to $V$ and $U$ is single-valued and we have
\begin{equation}\label{eqn:sec_gen}
h(x)\geq h(\tilde{x})+\langle \tilde{v},x-\tilde{x}\rangle+\kappa |x-\tilde{x}|^2~~ \textrm{ for all } x\in U,
\end{equation}
and all $(\tilde{x},\tilde{v})\in (U\times V)\cap\gph \partial f$. Shrinking $U$ and $V$, we may assume that the inclusion $U\subset {\bf B}_{\epsilon}(\bar{x})$ holds. Observe by Proposition~\ref{prop:wchar}, we have 
$$\argmin_{|x-\bar{x}|\leq\epsilon}\{f(x)-\langle \tilde{v},x\rangle\}={\bf B}_{\epsilon}(\bar{x})\cap(\partial f)^{-1}(\tilde{v}),$$
for all $\tilde{v}$ sufficiently close to $0$. In particular, we may shrink $V$ so that for all pairs $(\tilde{x},\tilde{v})\in (U\times V)\cap\gph \partial f$, we have $\partial_{co} f(\tilde{x})\neq\emptyset$. Then applying Lemma~\ref{lem:conv}, we deduce $h(\tilde{x})=f(\tilde{x})$ for all $(\tilde{x},\tilde{v})\in (U\times V)\cap\gph \partial f$, and $h(x)\leq f(x)$ for all $x\in U$. Plugging these relations into (\ref{eqn:sec_gen}), the result follows immediately.


$\ref{so}\Rightarrow\ref{stable}$: This follows directly from the definitions of prox-regularity and stable strong local minimizers.

$\ref{stable}\Rightarrow\ref{tilt}$: This implication is immediate from Proposition~\ref{prop:imp}.
\qed
\end{proof}

It is important to note that subdifferential continuity plays an important role in the validity of Theorem~\ref{thm:main}, as the following example shows. \begin{exa}[Failure of subdifferential continuity]
{\rm
Consider the function $f$ on $\R$ defined by 
$$f(x)= \left\{
        \begin{array}{ll}
            1+x^4, & \quad x < 0, \\
            x^2, & \quad x \geq 0.
        \end{array}
    \right.$$ 
One can easily check that $f$ is prox-regular at $\bar{x}=0$ for $\bar{v}=0$ and that the origin is a stable strong local minimizer of $f$. However $\partial f$ fails to be strongly metrically regular at $(0,0)$. This occurs, of course, because $f$ is not subdifferentially continuous at $\bar{x}=0$ for $\bar{v}=0$.
}
\end{exa}

The following example shows that tilt stability does not necessarily imply that prox-regularity holds. Hence the assumption of prox-regularity in conditions \ref{tilt} and \ref{stable} of Theorem~\ref{thm:main} is not superfluous.
\begin{exa}[Failure of prox-regularity]
{\rm
Consider the continuous function 
$f$ on $\R$ defined by 
$$f(x)= \left\{
        \begin{array}{ll}
            {\Big\lfloor {\frac{1}{|x|}}\Big\rfloor}^{-1}, & \quad x\neq 0, \\
            0, & \quad x = 0.
        \end{array}
    \right.$$ 
Clearly $\bar{x}=0$ gives a tilt-stable local minimum of $f$. Observe however $(\partial f)^{-1}(0)=\R$ and hence the subdifferential mapping 
$\partial f$ is not strongly metrically regular at $(0,0)$. Of course, this situation occurs because $f$ is not prox-regular at $\bar{x}$ for $0$. 
}
\end{exa}

\begin{rem}
{\rm 
In light of \cite[Theorem 1.3]{PR}, we may add another equivalence to Theorem~\ref{thm:main}, namely that $f$ is prox-regular at $\bar{x}$ for $0$ and the generalized Hessian mapping $\partial^{2} f(\bar{x}|0)$ is positive definite in the sense that 
$$\langle z, w\rangle >0 \textrm{ whenever } z\in \partial^{2} f(\bar{x}|0)(w),~ w\neq 0.$$ Hence in concrete instances, we may use the newly developed calculus of the generalized Hessians \cite{sorder} and the calculus of prox-regularity \cite{calc_prox} to determine when any of the equivalent properties listed in Theorem~\ref{thm:main} hold.

}
\end{rem}




\begin{acknowledgements}
The authors are indebted to Alexander D. Ioffe for bringing to their attention a possible relationship between quadratic growth conditions and metric regularity, in absence of convexity. The authors would also like to thank Boris Mordukhovich,  J\.{i}\v{r}\'{i} V. Outrata, and Francisco J. Arag\'{o}n Artacho for insightful discussions. 
\end{acknowledgements}



\bibliographystyle{plain}
\small
\parsep 0pt
\bibliography{dim_graph}

%
%

\end{document}